\documentclass{article}
\usepackage{fullpage}
\usepackage{epsfig}
\usepackage{graphics}
\usepackage{latexsym}
\usepackage{amsmath}
\usepackage{amsfonts}
\usepackage{amssymb}
\usepackage{mathrsfs}
\usepackage{pifont}
\usepackage{yhmath}
\usepackage{undertilde}
\usepackage{bbm}
\usepackage{dsfont}
\usepackage{eufrak}
\usepackage{amsthm}
\usepackage[usenames,dvipsnames]{color}
\usepackage{stmaryrd}
\usepackage{wasysym}
\usepackage{bbm}

\newtheorem{theorem}{Theorem}[section]
\newtheorem{lemma}[theorem]{Lemma}

\theoremstyle{remark}

\newtheorem{claim}{Claim}

\begin{document}
\newcounter{my}
\newenvironment{mylabel}
{
\begin{list}{(\roman{my})}{
\setlength{\parsep}{-1mm}
\setlength{\labelwidth}{8mm}
\usecounter{my}}
}{\end{list}}

\newcounter{my2}
\newenvironment{mylabel2}
{
\begin{list}{(\alph{my2})}{
\setlength{\parsep}{-0mm} \setlength{\labelwidth}{8mm}
\setlength{\leftmargin}{3mm}
\usecounter{my2}}
}{\end{list}}

\newcounter{my3}
\newenvironment{mylabel3}
{
\begin{list}{(\alph{my3})}{
\setlength{\parsep}{-1mm}
\setlength{\labelwidth}{8mm}
\setlength{\leftmargin}{10mm}
\usecounter{my3}}
}{\end{list}}

\title{\bf On the feedback number of $3-$uniform hypergraph}

\author{Zhuo Diao ${}^{a}$\thanks{Corresponding author. E-mail: diaozhuo@amss.ac.cn}\quad Zhongzheng Tang ${}^{b,c,d}$
}
\date{
 ${}^a$ School of Statistics and Mathematics, Central University of Finance and Economics
Beijing 100081, China\\
${}^b$ Academy of Mathematics and Systems Science, Chinese Academy of Sciences\\
 Beijing 100190, China\\
 ${}^c$ School of Mathematical Sciences, University of Chinese Academy of Sciences\\
 Beijing 10049, China\\
 ${}^d$ Department of Computer Science, City University of Hong Kong\\
 Hong Kong, China 
}
\maketitle

\begin{abstract}
Let $H=(V,E)$ be a hypergraph with vertex set $V$ and edge set $E$. $S\subseteq V$ is a feedback vertex set (FVS) of $H$ if $H\setminus S$ has no cycle and $\tau_c(H)$ denote the minimum cardinality of a FVS of $H$. In this paper, we prove $(i)$ if $H$ is a linear $3$-uniform hypergraph with $m$ edges, then $\tau_c(H)\le m/3$. $(ii)$ if $H$ is a $3$-uniform hypergraph with $m$ edges, then $\tau_c(H)\le m/2$ and furthermore, the equality holds on if and only if every component of $H$ is a $2-$cycle.\\

Let $H=(V,E)$ be a hypergraph with vertex set $V$ and edge set $E$. $A\subseteq E$ is a feedback edge set (FES) of $H$ if $H\setminus A$ has no cycle and $\tau_c'(H)$ denote the minimum cardinality of a FES of $H$. In this paper, we prove if $H$ is a $3$-uniform hypergraph with $p$ components, then $\tau_c'(H)\le 2m-n+p$.

\end{abstract}



\section{Hypergraphs}

Let $H=(V,E)$ be a hypergraph with vertex set $V$ and edge set $E$. Hypergraph $H$ is called {\em linear} if $|e\cap f|\le1$ for any distinct edges $e,f\in E$, and {$3$-uniform if $|e|=3$ for each $e\in E$. Let $k\geq 2$ be an integer. A cycle of length $k$ is a sequence $\{v_{1}e_{1}v_{2}e_{2}...v_{k}e_{k}v_{1}\}$ with: $(1) \{e_{1}, e_{2},..., e_{k}\}$ are distinct edges of $H$. $(2) \{v_{1}, v_{2},..., v_{k}\}$ are distinct vertices of $H$. $(3) \{v_{i},v_{i+1}\}\subseteq e_{i}$ for each $i\in [k]$, here $v_{k+1}=v_{1}$.~\cite{berge1989}\\\\
We say $S\subseteq V$ is a feedback vertex set (FVS) of $H$ if $H\setminus S$ has no cycle and $\tau_c(H)$ denote the minimum cardinality of a FVS of $H$.\\
We say $A\subseteq E$ is a feedback edge set (FES) of $H$ if $H\setminus A$ has no cycle and $\tau_c'(H)$ denote the minimum cardinality of a FES of $H$.\\
Obviously for any hypergraph $H=(V,E)$, we have $\tau_c(H)\leq \tau_c'(H)$.\\
There are many literatures about the feedback vertex/edge number in hypergraph theory.~\cite{berge1989}\cite{Brualdi2009}\cite{diestel2005}\cite{BVB1999}\cite{EP1965}\cite{FGPR2008}\cite{FPRM2000}


\section{The feedback vertex number}

\subsection{The bound for linear hypergraph}

In this subsection, for linear $3$-uniform hypergraph, we give the bound of the feedback vertex number.

\begin{theorem}\label{1/3}
Let $H$ be a linear $3$-uniform hypergraph with $m$ edges. Then $\tau_c(H)\le m/3$.
\end{theorem}
\begin{proof}
Suppose the theorem fails. Let us take out the counterexample $H=(V,E)$ with minimum number of edges, thus $\tau_c(H)> m/3$.
Obviously $m\geq 3$ and without loss of generality, we can assume $H$ has no isolated vertex. We will break the proof into a series of claims.

\begin{claim}\label{cycle}
  Every edge in $E$ is contained in some cycle in $H$.
\end{claim}

If there exists $e\in E$ which doesn't belong to any cycle of $H$. Then $\tau_c(H)= \tau_c(H\setminus e)$. Because $H=(V,E)$ is the counterexample with minimum number of edges, we have $\tau_c(H\setminus e)\leq (m-1)/3$. Thus $\tau_c(H)\leq m/3$, this is a contradiction with $\tau_c(H)> m/3$.

\begin{claim}\label{degree2}
 $\forall v\in H, d(v)\leq 2$.
\end{claim}

If there exists $v\in V$ with $d(v)\geq3$ in $H$, then $\tau_c(H\setminus v)\leq (m-d(v))/3\leq (m-3)/3$ because $H=(V,E)$ is the counterexample with minimum number of edges. Considering a minimum FVS $S$ of $H'=H\setminus v$, we have $S\subseteq V\setminus v$ and $|S|\leq (m-3)/3$. Thus $S\cup v$ is a FVS for $H$ and $|S\cup v|= |S|+1\leq m/3$, this is a contradiction with $\tau_c(H)> m/3$.

\begin{claim}\label{trianglefree}
 $H$ is triangle free.
\end{claim}

If there exists $C=v_{1}e_{1}v_{2}e_{2}v_{3}e_{3}v_{1}$ in $H$, then $\tau_c(H\setminus \{e_{1},e_{2},e_{3}\})\leq (m-3)/3$ because $H=(V,E)$ is the counterexample with minimum number of edges. Considering a minimum FVS $S$ of $H'=H\setminus \{e_{1},e_{2},e_{3}\}$,  we have $S\subseteq V$ and $|S|\leq (m-3)/3$. Due to claim \ref{degree2}, we have $S\cup v_{1}$ is a FVS for $H$ and $|S\cup v_{1}|\leq |S|+1\leq m/3$, this is a contradiction with $\tau_c(H)> m/3$.

\begin{claim}\label{2regular}
 $H$ is $2$-regular.
\end{claim}

Due to claim \ref{degree2}, we just need to prove there is no vertex with degree $1$. If there exists $v\in V$ with $d(v)= 1$ in $H$, we could assume $v\in e_{1}$. Due to claim \ref{cycle} and \ref{trianglefree}, we can assume there exists $C=v_{1}e_{1}v_{2}e_{2}v_{3}e_{3}v_{4}...e_{k}v_{1}$ with $k\geq 4$ in $H$ and $e_{1}= \{v_{1},v,v_{2}\}$. Then $\tau_c(H\setminus \{e_{1},e_{2},e_{3}\})\leq (m-3)/3$ because $H=(V,E)$ is the counterexample with minimum number of edges. Considering a minimum FVS $S$ of $H'=H\setminus \{e_{1},e_{2},e_{3}\}$,  we have $S\subseteq V$ and $|S|\leq (m-3)/3$.  Due to claim \ref{degree2} and $d(v)=1$, we have $S\cup v_{3}$ is a FVS for $H$ and $|S\cup v_{3}|\leq |S|+1\leq m/3$, this is a contradiction with $\tau_c(H)> m/3$.

\begin{claim}\label{4cyclefree}
 There is no cycle with length $4$ in $H$.
\end{claim}

Suppose on the contrary, there exists $C=v_{1}e_{1}v_{2}e_{2}v_{3}e_{3}v_{4}e_{4}v_{1}$ in $H$, we have $e_{1}\cap e_{3}= e_{2}\cap e_{4}=\emptyset$ due to claim \ref{trianglefree}. We can assume $e_{1}=\{v_{1},u_{1},v_{2}\}, e_{2}=\{v_{2},u_{2},v_{3}\}, e_{3}=\{v_{3},u_{3},v_{4}\}, e_{4}=\{v_{4},u_{4},v_{1}\}$ and these vertices are distinct. Due to claim \ref{2regular}, we can assume $u_{1}\in e_{5}\neq e_{1}$, $u_{2}\in e_{6}\neq e_{2}$ and $u_{3}\in e_{7}\neq e_{3}$. Due to claim \ref{trianglefree},we have $e_{5}\neq e_{6}, e_{6}\neq e_{7}$. 

If $e_{5}= e_{7}$, we have $\tau_c(H\setminus \{e_{1},e_{2},e_{3},e_{4},e_{5},e_{6}\})\leq (m-6)/3$ because $H=(V,E)$ is the counterexample with minimum number of edges. Considering a minimum FVS $S$ of $H'=H\setminus \{e_{1},e_{2},e_{3},e_{4},e_{5},e_{6}\}$,  we have $S\subseteq V$ and $|S|\leq (m-6)/3$. Due to claim \ref{2regular}, we have $S\cup u_{2}\cup u_{4}$ is a FVS for $H$ and $|S\cup u_{2}\cup u_{4}|\leq |S|+2\leq m/3$, this is a contradiction with $\tau_c(H)> m/3$. See Figure~\ref{4cycle} for an illustration.

If $e_{5}\neq e_{7}$, we have $\tau_c(H\setminus \{e_{1},e_{2},e_{3},e_{4},e_{5},e_{7}\})\leq (m-6)/3$ because $H=(V,E)$ is the counterexample with minimum number of edges. Considering a minimum FVS $S$ of $H'=H\setminus \{e_{1},e_{2},e_{3},e_{4},e_{5},e_{7}\}$,  we have $S\subseteq V$ and $|S|\leq (m-6)/3$. Due to claim \ref{2regular}, we have $S\cup u_{1}\cup u_{3}$ is a FVS for $H$ and $|S\cup u_{1}\cup u_{3}|\leq |S|+2\leq m/3$, this is a contradiction with $\tau_c(H)> m/3$. See Figure~\ref{4cycle} for an illustration.\\

\begin{figure}[h]
\begin{center}
\includegraphics[scale=0.4]{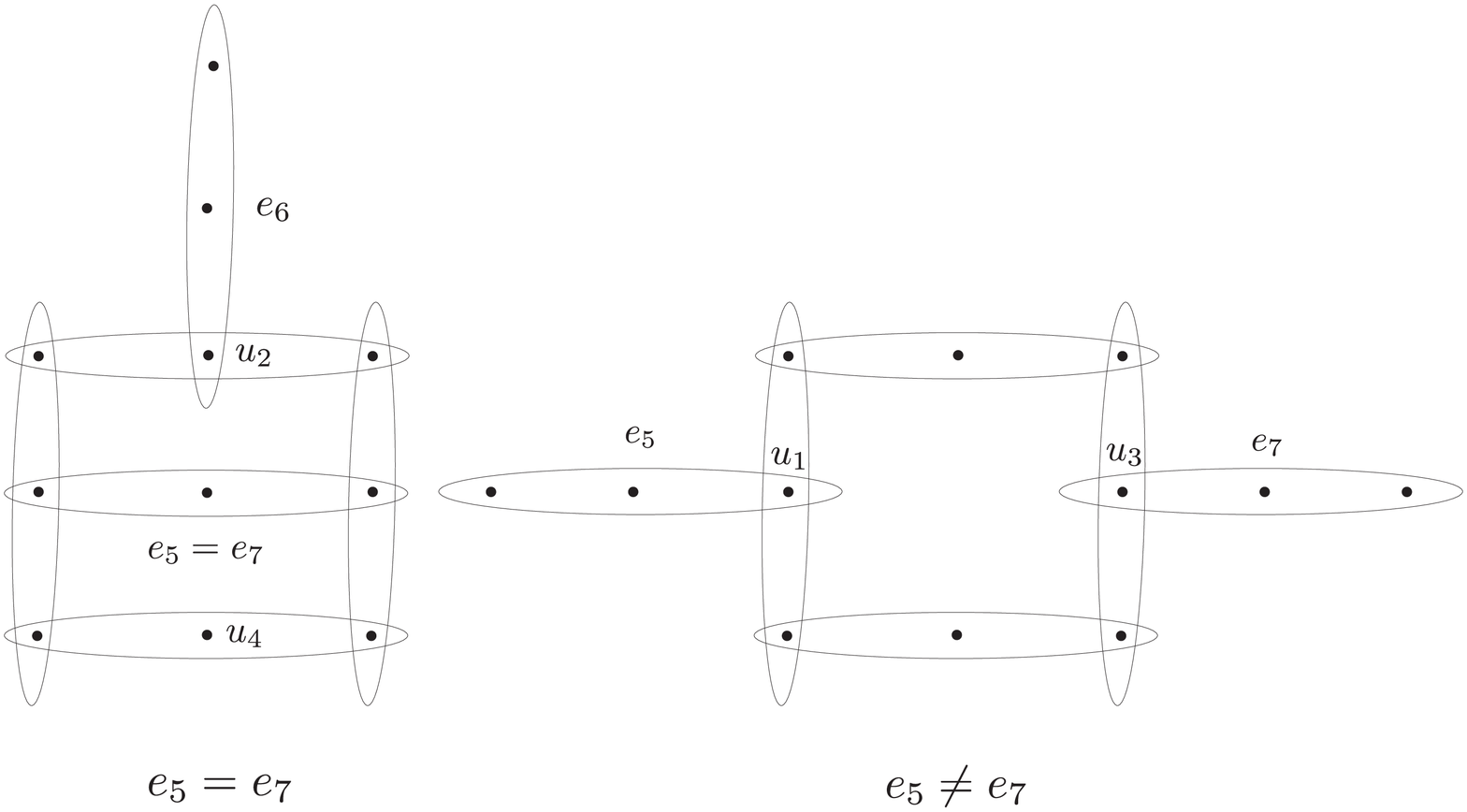}
\caption{\label{4cycle}{There is no cycle with length $4$ in $H$.}}
\end{center}
\end{figure}

Let $C=(V',E')=v_1e_1v_2e_2\ldots v_ke_kv_1$ be a shortest cycle in $H$. For each $i\in[k]$, suppose that $e_i=\{v_i,u_i,v_{i+1}\}$, where $v_{k+1}=v_1$. Due to claim \ref{trianglefree} and \ref{4cyclefree}, we have $k\geq 5$. Because $C$ is the shortest cycle, for each index pair $\{i\neq j\}\subseteq [k]$, if $e_{i}$ and $e_{j}$ are not adjacent in $C$, we have $e_{i}\cap e_{j}=\emptyset$. 

\begin{claim}\label{0mod3}
 $k\not\equiv0\pmod3$
\end{claim}

If $k\equiv0\pmod3$, we can assume $k=3t$ with $t\geq 2$. Then $\tau_c(H\setminus E')\leq (m-3t)/3$ because $H=(V,E)$ is the counterexample with minimum number of edges. Considering a minimum FVS $S$ of $H'=H\setminus E'$,  we have $S\subseteq V$ and $|S|\leq (m-3t)/3$. Due to claim \ref{2regular}, we have $S\cup {v_3,v_6,\ldots,v_{3i},\ldots,v_{3t}}$ is a FVS for $H$ and $|S\cup {v_3,v_6,\ldots,v_{3i},\ldots,v_{3t}}|\leq |S|+t\leq m/3$, this is a contradiction with $\tau_c(H)> m/3$. See Figure~\ref{kcycle} for an illustration.

\begin{claim}\label{1mod3}
 $k\not\equiv1\pmod3$
\end{claim}

If $k\equiv1\pmod3$, we can assume $k=3t+1$ with $t\geq 2$. Due to claim \ref{trianglefree}, \ref{2regular} and \ref{4cyclefree}, we have $u_{1}\in e_{3t+2}\neq e_{1},u_{3}\in e_{3t+3}\neq e_{3}$ and $e_{1},e_{2},e_{3}...e_{3t+1},e_{3t+2},e_{3t+3}$ are distinct. Then $\tau_c(H\setminus\{E',e_{3t+2},e_{3t+3}\}\leq (m-3t-3)/3$ because $H=(V,E)$ is the counterexample with minimum number of edges. Considering a minimum FVS $S$ of $H'=H\setminus\{E',e_{3t+2},e_{3t+3}\}$,  we have $S\subseteq V$ and $|S|\leq (m-3t-3)/3$. Due to claim \ref{2regular}, we have $S\cup {u_1,u_3,v_6,\ldots,v_{3i},\ldots,v_{3t}}$ is a FVS for $H$ and $|S\cup {u_1,u_3,v_6,\ldots,v_{3i},\ldots,v_{3t}}|\leq |S|+t+1\leq m/3$, this is a contradiction with $\tau_c(H)> m/3$. See Figure~\ref{kcycle} for an illustration.

\begin{claim}\label{2mod3}
 $k\not\equiv2\pmod3$
\end{claim}

If $k\equiv2\pmod3$, we can assume $k=3t+2$ with $t\geq 1$. Due to claim \ref{2regular} and the shortest cycle of $C$ , we have $u_{1}\in e_{3t+3}\neq e_{1}$ and $e_{1},e_{2},e_{3}...e_{3t+1},e_{3t+2},e_{3t+3}$ are distinct. Then $\tau_c(H\setminus\{E',e_{3t+3}\}\leq (m-3t-3)/3$ because $H=(V,E)$ is the counterexample with minimum number of edges. Considering a minimum FVS $S$ of $H'=H\setminus\{E',e_{3t+3}\}$,  we have $S\subseteq V$ and $|S|\leq (m-3t-3)/3$. Due to claim \ref{2regular}, we have $S\cup {u_1,v_4,\ldots,v_{3i+1},\ldots,v_{3t+1}}$ is a FVS for $H$ and $|S\cup {u_1,v_4,\ldots,v_{3i+1},\ldots,v_{3t+1}}|\leq |S|+t+1\leq m/3$, this is a contradiction with $\tau_c(H)> m/3$. See Figure~\ref{kcycle} for an illustration.\\

\begin{figure}[h]
\begin{center}
\includegraphics[scale=0.4]{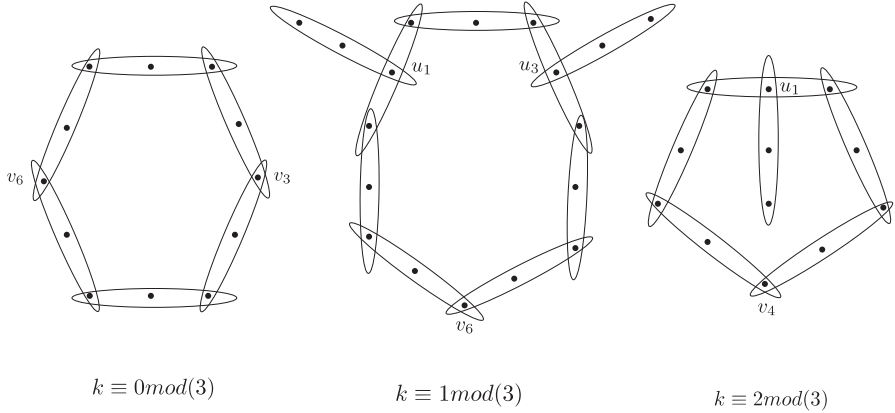}
\caption{\label{kcycle}{$k\equiv0,1,2\pmod3$.}}
\end{center}
\end{figure}

Above all, we have claim \ref{0mod3},\ref{1mod3} and \ref{2mod3}, this is impossible. Thus our assumption doesn't hold on. We finish our proof.
  \end{proof}

See Figure~\ref{sharp} for illustrations of five linear 3-uniform  hypergraphs attaining the upper bound.  It is easy to prove that the maximum degree of every extremal hypergraph (those $H$ with $\tau_c(H)= m/3$) is at most three. It would be interesting to characterize all extremal hypergraphs for Theorem \ref{1/3}.

\begin{figure}[h]
\begin{center}
\includegraphics[scale=0.65]{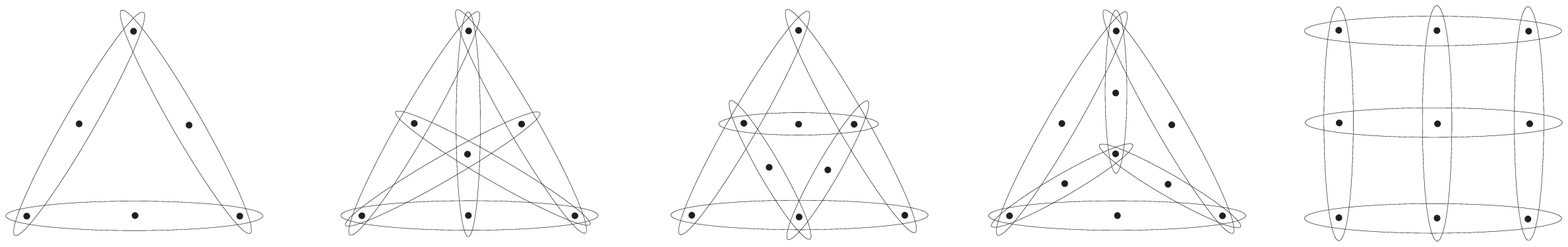}
\caption{\label{sharp}{Extremal linear 3-uniform hypergraphs $H$ with $\tau_c(H)= m/3$.}}
\end{center}
\end{figure}


\subsection{The bounds for general hypergraph}

In this subsection, for general $3$-uniform hypergraph, we give the bound of the feedback vertex number.

\begin{theorem}\label{1/2}
Let $H$ be a $3$-uniform hypergraph with $m$ edges. Then $\tau_c(H)\le m/2$.
Furthermore, the equality holds on if and only if every component of $H$ is a $2-$cycle.
\end{theorem}

\begin{proof}
First, we will prove $\tau_c(H)\le m/2$. Suppose the theorem fails. Let us take out the counterexample $H=(V,E)$ with minimum number of edges, thus $\tau_c(H)> m/2$. Without loss of generality, we can assume $H$ has no isolated vertex. We will get a contradiction through two claims.

\begin{claim}\label{cycle1}
  Every edge in $E$ is contained in some cycle in $H$.
\end{claim}

If there exists $e\in E$ which doesn't belong to any cycle of $H$. Then $\tau_c(H)= \tau_c(H\setminus e)$. Because $H=(V,E)$ is the counterexample with minimum number of edges, we have $\tau_c(H\setminus e)\leq (m-1)/2$. Thus $\tau_c(H)< m/2$, this is a contradiction with $\tau_c(H)> m/2$.

\begin{claim}\label{degree1a}
 $\forall v\in H, d(v)\leq 1$.
\end{claim}

If there exists $v\in V$ with $d(v)\geq 3$ in $H$, then $\tau_c(H\setminus v)\leq (m-d(v))/2\leq (m-2)/2$ because $H=(V,E)$ is the counterexample with minimum number of edges. Considering a minimum FVS $S$ of $H'=H\setminus v$, we have $S\subseteq V\setminus v$ and $|S|\leq (m-2)/2$. Thus $S\cup v$ is a FVS for $H$ and $|S\cup v|= |S|+1\leq m/2$, this is a contradiction with $\tau_c(H)> m/2$.

The above two claims lead to a contradiction immediately. \\\\
Second, we will prove the equality holds on if and only if every component of $H$ is a $2-$cycle. \\
Sufficiency: This is obvious.\\
Necessity: Let $H=(V,E)$ be a $3$-uniform hypergraph with $\tau_c(H)= m/2$. Without loss of generality, we can assume $H$ is connected and we just need to prove $H$ is a $2-$cycle. We will finish the proof through a series of claims.

\begin{claim}\label{cycle2}
  Every edge in $E$ is contained in some cycle in $H$.
\end{claim}

If there exists $e\in E$ which doesn't belong to any cycle of $H$. Then $\tau_c(H)= \tau_c(H\setminus e)$. Because $H=(V,E)$ is the counterexample with minimum number of edges, we have $\tau_c(H\setminus e)\leq (m-1)/2$. Thus $\tau_c(H)< m/2$, this is a contradiction with $\tau_c(H)= m/2$.

\begin{claim}\label{degree2b}
 $\forall v\in H, d(v)\leq 2$.
\end{claim}

If there exists $v\in V$ with $d(v)\geq 3$ in $H$, then $\tau_c(H\setminus v)\leq (m-d(v))/2\leq (m-3)/2$ because $H=(V,E)$ is the counterexample with minimum number of edges. Considering a minimum FVS $S$ of $H'=H\setminus v$, we have $S\subseteq V\setminus v$ and $|S|\leq (m-3)/2$. Thus $S\cup v$ is a FVS for $H$ and $|S\cup v|= |S|+1\leq (m-1)/2< m/2$, this is a contradiction with $\tau_c(H)= m/2$.

\begin{claim}\label{degree2c}
 $\forall v\in H$, if $d(v)= 2$, then $\tau_c(H\setminus v)= (m-2)/2$.
\end{claim}

Combining the next inequality and $\tau_c(H)= m/2$, this claim is obvious.

\[\tau_c(H)\leq \tau_c(H\setminus v)+1\leq (m-2)/2+1= m/2\]

Now $\tau_c(H)= m/2$, according to Theorem~\ref{1/3}, $H$ is not linear, thus $H$ contains a $2-$cycle: $v_{1}e_{1}v_{2}e_{2}v_{1}$. If $H$ is not a $2-$cycle, by claim~\ref{degree2b}, we can assume $e_{1}=\{v_{1},u,v_{2}\},u\in e$, $\{e,e_{1},e_{2}\}$ are three distinct edges. By claim~\ref{degree2c}, $d(u)= 2, \tau_c(H\setminus u)= (m-2)/2$. Notice that in $H\setminus u$, $e_{2}$ contains two $1-$ degree vertices, thus $e_{2}$ can not be contained in any cycle. Thus, we have the next inequality, which is a contradiction with $\tau_c(H\setminus u)= (m-2)/2$.

\[\tau_c(H\setminus u)=\tau_c(H\setminus u\setminus e_{2})\leq (m-3)/2< (m-2)/2\]

Above all, we finish the theorem's proof.

\end{proof}

\section{The feedback edge number}

In this section, for general $3$-uniform hypergraph, we give the bound of the feedback edge number.

\begin{theorem}\label{cyclecover}
Let $H=(V,E)$ be a  $3$-uniform hypergraph with $p$ components, then $\tau_c'(H)\le 2m-n+p$.
\end{theorem}



Before proving the theorem above, we will prove a series of lemmas which are very useful.

\begin{lemma}\label{connectn2m1}
For every $3-$uniform connected hypergraph $H(V,E)$, $n\leq 2m+1$ holds on.
\end{lemma}

\begin{proof}
We prove this lemma by induction on $m$. When $m=0$, $H(V,E)$ is an isolate vertex, $n\leq 2m+1$ holds on.
Assume this lemma holds on for $m\leq k$. When $m=k+1$, take arbitrarily one edge $e$ and consider the subgraph $H\setminus e$.
obviously, $H\setminus e$ has at most three components. Assume $H\setminus e$ has $p$ components $H_{i}(V_{i},E_{i})$ and $n_{i}=|V_{i}|, m_{i}=|E_{i}|$ for each $i\in\{1,...,p\}$. Then by our induction, $n_{i}\leq 2m_{i}+1$ holds on. So we have

\begin{equation}
n=n_{1}+...n_{p}\leq 2m_{1}+...2m_{p}+p=2(m-1)+p=2m+p-2\leq 2m+1
\end{equation}

By induction, we finish our proof.
\end{proof}

\begin{lemma}\label{connectntree2m1}
For every $3-$uniform connected hypergraph $H(V,E)$, $n=2m+1$ if and only if $H$ is a hypertree.
\end{lemma}

\begin{proof}
sufficiency: if $H$ is a hypertree, we prove $n=2m+1$ by induction on $m$. When $m=0$, $H(V,E)$ is an isolate vertex, $n=2m+1$ holds on.
Assume this lemma holds on for $m\leq k$. When $m=k+1$, take arbitrarily one edge $e$ and consider the subgraph $H\setminus e$.
Because $H$ is a hypertree, $H\setminus e$ has exactly three components, denoted by $H_{i}(V_{i},E_{i})$ and $n_{i}=|V_{i}|, m_{i}=|E_{i}|$ for each $i\in\{1,2,3\}$. Because every component is a hypertree, thus by our induction, $n_{i}= 2m_{i}+1$ holds on. So we have

\begin{equation}
n=n_{1}+n_{2}+n_{3}=2m_{1}+2m_{2}+2m_{3}+3=2(m-1)+3=2m+1
\end{equation}

By induction, we finish the sufficiency proof.\\

necessity: We prove by contradiction. If $H$ is not a hypertree, $H$ contain a cycle $C$. Take arbitrarily one edge $e$ in $C$ and consider the subgraph $H\setminus e$. obviously, $H\setminus e$ has at most two components. Assume $H\setminus e$ has $p$ components $H_{i}(V_{i},E_{i})$ and $n_{i}=|V_{i}|, m_{i}=|E_{i}|$ for each $i\in\{1,...,p\}$. Then by lemma~\ref{connectn2m1}, $n_{i}\leq 2m_{i}+1$ holds on. So we have

\begin{equation}
n=n_{1}+...n_{p}\leq 2m_{1}+...2m_{p}+p=2(m-1)+p=2m+p-2\leq 2m< 2m+1
\end{equation}

which is a contradiction with $n=2m+1$. Thus $H$ is a hypertree and we finish our necessity proof.

\end{proof}



Next we will prove the main theorem:

\begin{theorem}\label{cyclecover}
Let $H=(V,E)$ be a $3$-uniform hypergraph with $p$ components, then $\tau_c'(H)\le 2m-n+p$.
\end{theorem}

\begin{proof}  Pick arbitrarily a minimum FES $A\subseteq E$, then $\tau_c'(H)=|A|$. Suppose that $H\setminus A$ contains exactly $k$ components $H_{i}=(V_{i},E_{i})$, $i=1,\ldots,k$. It follows from Lemma~\ref{connectntree2m1} that $n_{i}=2m_{i}+1$ for each $i\in [k]$. Thus $n=\sum_{i\in [k]}n_{i}= 2\sum_{i\in [k]}m_{i}+k= 2(m-\tau_c'(H))+k$, which means $2\tau_c'(H)= 2m-n+k$. To establish the lemma, it suffices to prove  $k\leq \tau_c'(H)+ p$.

 In case of  $\tau_c'(H)=0$, we have $A=\emptyset$ and $k=p=\tau_c'(H)+p$.  In case of  $\tau_c'(H)\ge 1$, suppose that $A=\{e_{1},...,e_{|A|}\}$. Because $A$ is a minimum FES of $H$,  for each $i\in [|A|]$, there is a cycle $C_{i}$ in $H\setminus (A\setminus \{e_{i}\})$ such that $e_{i}\in C_{i}$. Considering $H\setminus A$ being obtained from $H$ be removing $e_1,e_2,\ldots,e_{|A|}$ sequentially, for $i=1,\ldots,|A|$, since $|e_i|=3$, the presence of $C_i$ implies that the removal of $e_i$ can create at most one more component. Therefore we have  $k\le p+|A|$ as desired.
\end{proof}

\section{Conclusion and future work}

Conclusion: $(i)$ if $H$ is a linear $3$-uniform hypergraph with $m$ edges, then $\tau_c(H)\le m/3$. $(ii)$ if $H$ is a $3$-uniform hypergraph with $m$ edges, then $\tau_c(H)\le m/2$ and furthermore, the equality holds on if and only if every component of $H$ is a $2-$cycle. $(iii)$ if $H$ is a $3$-uniform hypergraph with $p$ components, then $\tau_c'(H)\le 2m-n+p$.\\\\
Future work: for linear $3-$uniform hypergraph, it would be interesting to characterize all extremal hypergraphs (those $H$ with $\tau_c(H)= m/3$).

\bibliography{all}

\end{document}